\documentclass[12pt]{article}
\usepackage{amsfonts}
\usepackage{graphicx}
\usepackage{amsmath}

\numberwithin{equation}{section}

\newtheorem{theorem}{Theorem}[section]

\newtheorem{corollary}[theorem]{Corollary}

\newtheorem{lemma}[theorem]{Lemma}

\newtheorem{proposition}[theorem]{Proposition}

\newenvironment{proof}[1][Proof]{\textbf{#1.} }{\ \rule{0.5em}{0.5em}}
\input{tcilatex}
\begin{document}

\title{On Girsanov's transform for backward stochastic differential equations%
}
\date{}
\author{{\small {\textsc{By G. Liang{,} {A. Lionnet }and {Z. Qian}}}} \\
\emph{{\small {Oxford-Man Institute, Oxford University, England}}}}
\maketitle

\leftskip1truecm \rightskip1truecm \noindent {\textbf{Abstract.}} By using a
simple observation that the density processes appearing in It\^{o}'s
martingale representation theorem are invariant under the change of
measures, we establish a non-linear version of the Cameron-Martin formula
for solutions of a class of systems of quasi-linear parabolic equations with
non-linear terms of quadratic growth. We also construct a local stochastic
flow and establish a Bismut type formula for such system of quasi-linear
PDEs. Gradient estimates are obtained in terms of the probability
representation of the solution. Another interesting aspect indicated in the
paper is the connection between the non-linear Cameron-Martin formula and a
class of forward-backward stochastic differential equations (FBSDEs).

\leftskip0truecm \rightskip0truecm

\vskip4truecm


\vskip0.5truecm


\newpage

\section{Introduction}

The goal of the article is to present a unified approach for Girsanov's
techniques of changing probability measures used in the recent literature on
backward stochastic differential equations (BSDE). The framework we
formulate has an advantage allowing us to bring together seemingly different
sorts of results on BSDE, forward-backward stochastic differential equations
(FBSDE), and function equations on a probability space.

Consider the following system of quasi-linear parabolic equations%
\begin{equation}
\frac{\partial }{\partial t}u^{i}=\frac{1}{2}\Delta
u^{i}+\sum_{j=1}^{d}f^{j}(u,\nabla u)\frac{\partial u^{i}}{\partial x^{j}}%
\text{ \ \ \ \ in }\mathbb{R}^{d}  \label{non-lk1}
\end{equation}%
$i=1,\cdots ,m$, where the drift vector field involves a solution $u$ and
its total derivative $\nabla u=(\frac{\partial u^{i}}{\partial x^{j}})$
through $f=(f^{j})_{j\leq d}$. In the interesting cases, $f^{j}:\mathbb{R}%
^{d}\times \mathbb{R}^{m\times d}\rightarrow \mathbb{R}$ ($j=1,\cdots ,d$)
are Lipschitz continuous but unbounded, and therefore the non-linear term
(often called the \emph{convection} term) appearing in (\ref{non-lk1}) is of
quadratic growth. This is a kind of non-linear feature which appears in many
physical PDEs, see for example \cite{MR833742} and \cite{MR748308}. If $%
m=d=3 $ and $f^{j}(u)=u^{j}$, then (\ref{non-lk1}) is a modification of the
Navier-Stokes equations with the pressure term and the divergence-free
condition dropped altogether, while the same non-linear convection term has
been retained. This kind of PDEs has been used as simplified models for
phenomena such as turbulence flows. Due to the special structure of the
system (\ref{non-lk1}), the maximum principle applies to $|u(x,t)|^{2}$,
thus a bounded solution exists as long as the initial data $u(x,0)$ is
regular and bounded, which makes a distinctive difference from the
Navier-Stokes equations. According to Theorem 7.1 on page 596 in \cite%
{MR0241822}, if the initial data $u_{0}$ is smooth and bounded with bounded
derivatives, then a bounded, smooth solution $u$ to the initial value
problem of (\ref{non-lk1}) exists for all time. Our main interest in this
article is to establish probabilistic representations for the solution $u$
by applying Girsanov's theorem to BSDEs.

To this end, it will be a good idea to look at Peng's non-linear Feymann-Kac
formula (see \cite{MR1149116}) for general quasi-linear parabolic equations
with Lipschitz non-linear term. The main idea in \cite{MR1037747} and \cite%
{MR1149116} can be described as following. Thanks to It\^{o}'s calculus,
Brownian motion $B=\left( B^{1},\cdots ,B^{d}\right) $ may be considered as
\textquotedblleft coordinates\textquotedblright\ on the space of continuous
paths equipped with the Wiener measure, and it is a much tested idea that
one may \textquotedblleft read out\textquotedblright\ solutions of
quasi-linear partial differential equations in terms of $B$. Suppose $%
u(x,t)=(u^{1}(x,t),\cdots ,u^{m}(x,t))$ is a sufficiently smooth solution to
the Cauchy problem of the system of quasi-linear parabolic equations 
\begin{equation*}
\left( \frac{1}{2}\Delta -\frac{\partial }{\partial t}\right)
u^{i}+h^{i}(u,\nabla u)=0\text{ \ \ \ \ in }\boldsymbol{R}^{d}\times \lbrack
0,\infty )
\end{equation*}%
with the initial data $u^{i}(x,0)=u_{0}^{i}(x)$. Applying It\^{o}'s formula
to $Y_{t}^{i}=u^{i}(x+B_{t},T-t)$ (and set $Z_{t}^{i,j}=\frac{\partial u^{i}%
}{\partial x^{j}}(x+B_{t},T-t)$) to obtain 
\begin{eqnarray*}
Y_{T}^{i}-Y_{t}^{i} &=&\sum_{j=1}^{d}\int_{t}^{T}\frac{\partial u^{i}}{%
\partial x^{j}}(x+B_{s},T-s)dB_{s}^{j} \\
&&+\int_{t}^{T}\left( \frac{1}{2}\Delta -\frac{\partial }{\partial s}\right)
u^{i}(x+B_{s},T-s)ds \\
&=&\sum_{j=1}^{d}\int_{t}^{T}Z_{s}^{i,j}dB_{s}^{j}-%
\int_{t}^{T}h^{i}(Y_{s},Z_{s})ds
\end{eqnarray*}%
for $t\in \lbrack 0,T]$. The pair $(Y,Z)$ is a solution to the stochastic
differential equations%
\begin{equation}
dY_{t}^{i}=-h^{i}(Y_{t},Z_{t})dt+\sum_{j=1}^{d}Z_{t}^{i,j}dB_{t}^{j}
\label{bsdf1}
\end{equation}%
on $(\Omega ,\mathcal{F},\mathcal{F}_{t},\boldsymbol{P})$ with the terminal
data $Y_{T}^{i}=u_{0}^{i}(x+B_{T})$.

Peng \cite{MR1149116} and Pardoux-Peng \cite{MR1037747} made two crucial
observations: firstly the process $Z_{t}$ may be recovered from the It\^{o}
representation of the martingale $S_{t}^{i}=S_{0}^{i}+\int_{0}^{t}%
\sum_{j=1}^{d}Z_{s}^{i,j}dB_{s}^{j}$. Secondly $S$ is indeed the martingale
part of the continuous semimartingale $Y$, therefore (\ref{bsdf1}) is a
closed system and may be solved by specifying a terminal value $Y_{T}$. It
was proved in Pardoux-Peng \cite{MR1037747} that if $h^{i}$ are global
Lipschitz continuous, then (\ref{bsdf1})\ may be solved as long as $Y_{T}\in
L^{2}(\Omega ,\mathcal{F}_{T},\boldsymbol{P})$.

In the case that $h^{i}(y,z)$ has a special form such as $%
\sum_{j=1}^{d}f^{j}(y,z)z^{ij}$, one is able to solve (\ref{bsdf1}) by
firstly solve the trivial BSDE $dY_{t}^{i}=%
\sum_{j=1}^{d}Z_{t}^{i,j}dB_{t}^{j}$ then change the probability measure,
even for non-Lipschitz non-linear term $h$ as long as $f$ is Lipschitz
continuous. We are thus able to extend the Cameron-Martin formula to a class
of \emph{systems} of quasi-linear parabolic equations with quadratic growth.
The non-linear version of the Cameron-Martin formula, which may be
considered as our contribution of this article, is of independent interest.

Suppose that the initial data $u_{0}$ is Lipschitz continuous and bounded. \
Let $(\mathcal{F}_{t})_{t\geq 0}$ be Brownian filtration, i.e. the
filtration over $\boldsymbol{W}^{d}$ of all continuous paths in $\boldsymbol{%
R}^{d}$ generated by the coordinate process $\{B_{t}:t\geq 0\}$, augmented
by the Wiener measure. Let $T>0$ and $x\in \boldsymbol{R}^{d}$ be fixed but
arbitrary. For each $\xi =(\xi ^{1},\cdots ,\xi ^{m})$ where $\xi ^{i}\in
L^{\infty }(\boldsymbol{W}^{d},\mathcal{F}_{T},\boldsymbol{P}^{x})$. Define 
\begin{equation*}
\tilde{B}(\xi )_{t}=B_{t}-\int_{0}^{t}f(Y(\xi )_{s},Z(\xi )_{s})ds
\end{equation*}%
where $Y(\xi )_{t}=E^{\boldsymbol{P}}(\xi |\mathcal{F}_{t})$ are bounded
martingales and $Z(\xi )=(Z(\xi )^{i,j})$ are determined by It\^{o}'s
martingale representation:%
\begin{equation*}
\xi ^{i}-E^{\boldsymbol{P}}(\xi ^{i}|\mathcal{F}_{0})=\sum_{j=1}^{d}%
\int_{0}^{T}Z(\xi )_{t}^{i,j}dB_{t}^{j}\text{.}
\end{equation*}%
We prove that there is a unique $\xi \in L^{\infty }(\boldsymbol{W}^{d},%
\mathcal{F}_{T},\boldsymbol{P}^{x})$ such that $\xi =u_{0}(\tilde{B}(\xi
)_{T})$, and $u(x,T)=E^{\boldsymbol{P}}(Y(\xi )_{0})$, which can be
considered as the non-linear version of the Cameron-Martin formula. More
information about the solution $u$ may be obtained as we can represent the
derivative $\frac{\partial u^{i}}{\partial x^{j}}$ in terms of the process $%
(Z(\xi )_{t}^{i,j})_{t\leq T}$, see Theorem \ref{main-t31} below. The
non-linear version of the Cameron-Martin formula may be reformulated in
terms of FBSDE as well, see (\ref{fbde1}) and\ Corollary \ref{coro-h2}.

The main reason we are interested in representations of solutions to
physical PDEs in terms of Brownian motion or in general in terms of It\^{o}%
's diffusions, lies in the fact that it is then possible to employ
probabilistic methods such as It\^{o}'s calculus, Malliavin's calculus of
variations and path integration method to the study of non-linear PDEs. We
demonstrate this point by deriving explicit gradient estimates for solutions
of a class of systems of quasi-linear parabolic equations.

Finally we would like to point out that the type of BSDEs such as (\ref%
{bsdf1}) with quadratic growth non-linear terms has been well studied in
Kobylanski \cite{MR1782267}. Her work has been extended and generalized
substantially in Briand and Hu \cite{MR2391164} and \cite{MR2257138}. BSDEs
with quadratic growth driven by martingales are solved by Morlais \cite%
{MR2465489} and Tevzadze \cite{MR2389055}. However their methods can be
applied to scalar BSDEs, not systems in general.

\section{Girsanov's theorem, martingale representation and BSDE}

Let $B=(B^{1},\cdots ,B^{d})$ be a Brownian motion in $\boldsymbol{R}^{d}$
started at $0$ on a complete probability space $(\Omega ,\mathcal{F},%
\boldsymbol{P})$, $(\mathcal{F}_{t})_{t\geq 0}$ be the filtration generated
by $B$ augmented by probability zero sets in $\mathcal{F}$ and $\mathcal{F}%
_{\infty }=\sigma \{\mathcal{F}_{t}:t\geq 0\}$. Since we are interested in $%
\mathcal{F}_{\infty }$-measurable random variables only, for simplicity, we
assume that $\mathcal{F}=\mathcal{F}_{\infty }$. According to It\^{o}'s
martingale representation theorem, any martingale $S$ on $(\Omega ,\mathcal{F%
},\mathcal{F}_{t},\boldsymbol{P})$ is continuous and has a unique
representation in terms of It\^{o}'s integration%
\begin{equation*}
S_{t}-S_{0}=\sum_{j=1}^{d}\int_{0}^{t}D_{B}(S)_{s}^{j}dB_{s}^{j}
\end{equation*}%
where $D_{B}(S)^{j}$ are predictable processes, called the density processes
of $S$ with respect to Brownian motion $B$. Since we will deal with several
equivalent measures on $(\Omega ,\mathcal{F})$ at the same time, it is
desirable to have labels associated with notations which involve a
probability measure. We will follow this convention if confusions may arise.
Therefore, $E^{\boldsymbol{P}}$ and $E^{\boldsymbol{P}}\left\{ \cdot |%
\mathcal{F}_{t}\right\} $ denote the expectation and conditional expectation
with respect to $\boldsymbol{P}$ respectively.

Let $\boldsymbol{Q}$ be a probability measure on $(\Omega ,\mathcal{F})$
equivalent to $\boldsymbol{P}$, whose Radon-Nikodym's derivative with
respect to $\boldsymbol{P}$ restricted on $\mathcal{F}_{t}$ is denoted by $%
R_{t}$, that is, $\left. \frac{d\boldsymbol{Q}}{d\boldsymbol{P}}\right\vert
_{\mathcal{F}_{t}}=R_{t}$ for $t\geq 0$. Then $R$ is a positive martingale
on $(\Omega ,\mathcal{F},\mathcal{F}_{t},\boldsymbol{P})$ with $E^{%
\boldsymbol{P}}(R_{t})=1$. Girsanov's theorem \cite{MR0133152} establishes a
correspondence between local $\boldsymbol{P}$-martingales and local $%
\boldsymbol{Q}$-martingales. If $X=(X_{t})_{t\geq 0}$ is a local martingale
under probability $\boldsymbol{P}$, then $X$ is a continuous semi-martingale
under $\boldsymbol{Q}$ and $\tilde{X}_{t}=X_{t}-\int_{0}^{t}\frac{1}{R_{s}}%
d\langle X,R\rangle _{s}$ is a local martingale under $\boldsymbol{Q}$,
where the bracket process $\langle X,R\rangle _{t}$ is defined under the
probability $\boldsymbol{P}$, which is however invariant under the change of
equivalent probability measure.

It is convenient to formulate the Girsanov's transform in terms of
exponential martingales. The exponential martingale $\mathcal{E}(N)_{t}=\exp
\{N_{t}-\frac{1}{2}\langle N\rangle _{t}\}$ of a continuous local martingale 
$N$ (under $\boldsymbol{P}$) is the unique solution to the stochastic
exponential equation 
\begin{equation*}
\mathcal{E}(N)_{t}=1+\int_{0}^{t}\mathcal{E}(N)_{s}dN_{s}\text{.}
\end{equation*}%
Up to an initial data $R_{0}$, the Radon-Nikodym derivative $%
R=(R_{t})_{t\geq 0}$ of $\boldsymbol{Q}$ with respect to the measure $%
\boldsymbol{P}$ must be an exponential martingale of some continuous local
martingale $N$, so that $R_{t}=R_{0}\exp \{N_{t}-\frac{1}{2}\langle N\rangle
_{t}\}$. Then $\tilde{X}=X-\langle X,N\rangle $ for any $\boldsymbol{P}$%
-local martingale $X$.

According to L\'{e}vy's theorem, $\tilde{B}=(\tilde{B}^{1},\cdots ,\tilde{B}%
^{d})$ is a Brownian motion under $\boldsymbol{Q}$ which is $(\mathcal{F}%
_{t})_{t\geq 0}$-adapted, but the natural filtration $(\mathcal{\tilde{F}}%
_{t})_{t\geq 0}$ of $\tilde{B}$ may not coincide with $(\mathcal{F}%
_{t})_{t\geq 0}$.\ It can happen that $\mathcal{\tilde{F}}_{t}$ is strictly
smaller than $\mathcal{F}_{t}$ for some $t$. But, nevertheless any $(%
\mathcal{F}_{t})$-martingale under $\boldsymbol{Q}$ may be represented as It%
\^{o}'s integral of $(\mathcal{F}_{t})$-predictable processes against $%
\tilde{B}$. The starting point of our approach is the following elementary
fact.

\begin{lemma}
\label{thr-1}For any local $\boldsymbol{P}$-martingale $X$ (with respect to
the filtration $(\mathcal{F}_{t})_{t\geq 0}$) $D_{B}(X)=D_{\tilde{B}}(\tilde{%
X})$. That is, the density process appearing in It\^{o}'s martingale
representation is invariant under the change of equivalent probability
measures.
\end{lemma}

The lemma follows directly from the definition of density processes and the
Girsanov's theorem, so its proof is left to the reader. With the help of
this simple fact, we can deal with the question of how a backward stochastic
differential equation is transformed under change of equivalent probability
measures.

\begin{theorem}
\label{thr-4}Suppose that $g$ is global Lipschitz continuous, $\xi \in
L^{2}(\Omega ,\mathcal{F}_{T},P)$, and $(Y,S)$ is the unique solution pair
to BSDE 
\begin{equation}
dY_{t}=-g(t,Y_{t},D_{B}(S)_{t})dt+dS_{t}\text{, }Y_{T}=\xi   \label{bs-d1}
\end{equation}%
on $(\Omega ,\mathcal{F},\mathcal{F}_{t},\boldsymbol{P})$. Let $\tilde{S}%
=S-\langle S,N\rangle $. Then $(Y,\tilde{S})$ solves 
\begin{equation}
dY_{t}=-g(t,Y_{t},D_{\tilde{B}}(\tilde{S})_{t})dt+\sum_{j=1}^{d}D_{\tilde{B}%
}(\tilde{S})_{t}^{j}D_{B}(N)_{t}^{j}dt+d\tilde{S}_{t}  \label{bs-d2}
\end{equation}%
on $(\Omega ,\mathcal{F},\mathcal{F}_{t},\boldsymbol{Q})$, $Y_{T}=\xi $.
\end{theorem}

\begin{proof}
Since $(Y,S)$ is a solution to (\ref{bs-d1}), $Y^{i}$ are continuous
semimartingale with martingale parts $S^{j}$. By Lemma \ref{thr-1}, $%
D_{B}(S^{i})=D_{\tilde{B}}(\tilde{S}^{i})$, so that 
\begin{equation*}
\tilde{S}_{t}^{i}-\tilde{S}_{0}^{i}=\sum_{j=1}^{d}\int_{0}^{t}D_{\tilde{B}}(%
\tilde{S}^{i})_{s}^{j}d\tilde{B}_{s}^{j}=\sum_{j=1}^{d}%
\int_{0}^{t}D_{B}(S^{i})_{s}^{j}d\tilde{B}_{s}^{j}
\end{equation*}%
and 
\begin{equation*}
\tilde{S}_{t}^{i}=S_{t}^{i}-\sum_{j=1}^{d}\int_{0}^{t}D_{\tilde{B}}(\tilde{S}%
^{i})_{s}^{j}D_{B}(N)_{s}^{j}ds\text{.}
\end{equation*}%
Therefore 
\begin{eqnarray}
Y_{T}^{i}-Y_{t}^{i} &=&-\int_{t}^{T}g^{i}(s,Y_{s},D_{\tilde{B}}(\tilde{S}%
)_{s})ds  \notag \\
&&+\sum_{j=1}^{d}\int_{t}^{T}D_{\tilde{B}}(\tilde{S}%
^{i})_{s}^{j}D_{B}(N)_{s}^{j}ds+\tilde{S}_{T}^{i}-\tilde{S}_{t}^{i}
\label{n-bs-1}
\end{eqnarray}%
which is valid under the probability $\boldsymbol{Q}$. That is, the pair $(Y,%
\tilde{S})$ solves BSDE (\ref{bs-d2}) on $(\Omega ,\mathcal{F},\mathcal{F}%
_{t},\boldsymbol{Q})$, with terminal values $Y_{T}^{i}=\xi ^{i}$.
\end{proof}

The most interesting case is the following special choice of $N$.

\begin{corollary}
\label{coro-h01} Under the same assumptions as in the previous theorem, and%
\begin{equation*}
N_{t}=\sum_{j=1}^{d}\int_{0}^{t}f^{j}(Y_{s},D_{B}(S)_{s})dB_{s}^{j}
\end{equation*}%
where $f:\mathbb{R}^{m}\times \mathbb{R}^{md}\rightarrow \mathbb{R}^{d}$ is
Borel measurable. Then $(Y,\tilde{S})$ solves the BSDE%
\begin{equation*}
dY_{t}=-g(t,Y_{t},D_{\tilde{B}}(\tilde{S})_{t})dt+\sum_{j=1}^{d}D_{\tilde{B}%
}(\tilde{S})_{t}^{j}f^{j}(Y_{t},D_{\tilde{B}}(\tilde{S})_{t})dt+d\tilde{S}%
_{t}
\end{equation*}%
with the terminal value $Y_{T}=\xi $ on $(\Omega ,\mathcal{F},\mathcal{F}%
_{t},\boldsymbol{Q})$.
\end{corollary}

One may reformulate Corollary \ref{coro-h01} in terms of forward-backward
stochastic differential equations. Observe that, with the choice of $N$ made
in Corollary \ref{coro-h01}, $X=x+\tilde{B}$ is the solution to the
stochastic differential equation%
\begin{equation*}
dX_{t}=-f(Y_{s},Z_{t})dt+dB_{t}\text{, \ \ }X_{t}=x
\end{equation*}%
while $(Y,Z=D_{B}(S))$ is the solution of the BSDE%
\begin{equation*}
dY_{t}=-g(t,Y_{t},Z_{t})dt+Z_{t}dB_{t}\text{, \ }Y_{T}=u_{0}(X_{T})\text{,}
\end{equation*}%
thus $(X,Y,Z)$ is a solution to the following forward-backward stochastic
differential equations (FBSDEs)%
\begin{equation}
\left\{ 
\begin{array}{cc}
dX_{t}=-f(Y_{s},Z_{t})dt+dB_{t}\text{, \ \ \ \ \ \ \ } &  \\ 
dY_{t}=-g(t,Y_{t},Z_{t})dt+Z_{t}dB_{t}\text{, \ \ } &  \\ 
X_{t}=x\text{, }Y_{T}=u_{0}(X_{T})\text{. \ \ \ \ \ \ \ \ \ \ \ \ } & 
\end{array}%
\right.  \label{fbde1}
\end{equation}

FBSDEs such as (\ref{fbde1}) have been studied by various authors, and are
well presented in the research monograph \cite{MR1704232}. By utilizing the
fundamental apriori estimates established in \cite{MR0241822}, it has been
proved in \cite{MR1262970} that FBSDE (\ref{fbde1}) has a unique solution
such that $Y$ is bounded, as long as $u_{0}$ is bounded and Lipschitz, and
if $f$ and $g$ are Lipschitz continuous.

\begin{corollary}
\label{coro-h2}Let $f$ and $g$ be Lipschitz continuous, $u_{0}$ be bounded
and Lipschitz continuous. Let $(X,Y,Z)$ be the unique solution of (\ref%
{fbde1}) such that $Y$ is bounded. Define $\boldsymbol{Q}$ on $(\Omega ,%
\mathcal{F}_{T})$ by 
\begin{equation*}
\left. \frac{d\boldsymbol{Q}}{d\boldsymbol{P}}\right\vert _{\mathcal{F}%
_{T}}=\exp \left\{ \int_{0}^{T}f(Y_{s},Z_{s}).dB_{s}-\frac{1}{2}%
\int_{0}^{T}|f(Y_{s},Z_{s})|^{2}ds\right\} \text{.}
\end{equation*}%
Then $(Y,\tilde{Z})$ is the unique solution (such that $Y$ is bounded) to%
\begin{equation}
\left\{ 
\begin{array}{cc}
dY_{t}=\left[ -g(t,Y_{t},\tilde{Z}_{t})+\sum_{j=1}^{d}f^{j}(Y_{t},\tilde{Z}%
_{t})\tilde{Z}_{t}^{j}\right] dt+\sum_{j=1}^{d}\tilde{Z}_{t}^{j}d\tilde{B}%
_{t}^{j}\text{,} &  \\ 
Y_{T}=u_{0}(X_{T})\text{ \ \ \ \ \ \ \ \ \ \ \ \ \ \ \ \ \ \ \ \ \ \ \ \ \ \
\ \ \ \ \ \ \ \ \ \ \ \ \ \ \ \ \ \ \ \ \ \ \ \ \ \ \ \ \ \ \ \ \ \ \ \ \ }
& 
\end{array}%
\right.  \label{fbde2}
\end{equation}%
on $(\Omega ,\mathcal{F},\mathcal{F}_{t},\boldsymbol{Q})$.
\end{corollary}

\section{Cameron-Martin formula for PDEs}

In this section we apply the machinery developed in the previous section to
the initial value problem of the following quasi-linear parabolic system 
\begin{equation}
\frac{\partial u^{i}}{\partial t}+\sum_{j=1}^{d}f^{j}(u,\nabla u)\frac{%
\partial u^{i}}{\partial x^{j}}=\frac{1}{2}\Delta u^{i}+g^{i}(u,\nabla u)%
\text{ \ \ in }\mathbb{R}^{d}  \label{peq-01}
\end{equation}%
for $i=1,\cdots ,m$, where $f^{j}$ and $g^{i}$ are global Lipschitz
continuous in their arguments, together with the initial data $u(x,\cdot
)=u_{0}(x)$ which is bounded and global Lipschitz continuous. According to
Theorem 7.1 in \cite{MR0241822}, there is a unique bounded solution $u(x,t)$
to the initial value problem of (\ref{peq-01}). By It\^{o}'s formula, $%
Y_{t}=u(x+B_{t},T-t)$ and $Z_{t}=\nabla u(x+B_{t},T-t)$ solve the BSDE%
\begin{equation}
dY_{t}^{i}=\left[ -g^{i}(Y_{t},Z_{t})+%
\sum_{j=1}^{d}f^{j}(Y_{t},Z_{t})Z_{t}^{i,j}\right] dt+%
\sum_{j=1}^{d}Z_{t}^{i,j}dB_{t}^{j}  \label{bsde-q1}
\end{equation}%
with the terminal value $Y_{T}=u_{0}(x+B_{T})$ over the filtered probability
space $(\Omega ,\mathcal{F},\mathcal{F}_{t},\boldsymbol{P})$, which is
indeed the unique bounded solution of BSDE (\ref{bsde-q1}).

On the other hand, since $g^{i}$ are global Lipschitz, according to Pardoux
and Peng \cite{MR1037747}, for any $\xi \in L^{2}(\Omega ,\mathcal{F}_{T},%
\boldsymbol{P})$ there is a unique solution pair $(Y,Z)$ to the backward
stochastic differential equation:%
\begin{equation}
dY_{t}=-g(Y_{t},Z_{t})dt+\sum_{j=1}^{d}Z_{t}^{j}dB_{t}^{j}\text{, }Y_{T}=\xi 
\label{u-0w1}
\end{equation}%
on $(\Omega ,\mathcal{F},\mathcal{F}_{t},\boldsymbol{P})$, which is the
system derived from (\ref{bsde-q1}) with the quadratic non-linear term
dropped . The solution $(Y,Z)$ depends on the terminal value $\xi $, so it
is denoted by $Y(\xi )$ and $Z(\xi )$ respectively. For $x\in \boldsymbol{R}%
^{d}$ and $T>0$, let 
\begin{equation}
N(\xi )_{t}=\sum_{j=1}^{d}\int_{0}^{t}f^{j}(Y(\xi )_{s},Z(\xi
)_{s})dB_{s}^{j}  \label{bb-qa2}
\end{equation}%
and 
\begin{equation}
\tilde{B}(\xi )_{t}=B_{t}-\int_{0}^{t}f(Y(\xi )_{s},Z(\xi )_{s})ds\text{.}
\label{ca-e1}
\end{equation}

\begin{theorem}
\label{main-t31}Let $u$ be a classical solution of (\ref{peq-01}) on $%
[0,T]\times \boldsymbol{R}^{d}$ such that $u(0,\cdot )=u_{0}(\cdot )$.
Suppose that $\xi \in L^{\infty }(\Omega ,\mathcal{F}_{T},\boldsymbol{P})$
is the solution to the function equation%
\begin{equation*}
\xi =u_{0}(x+\tilde{B}(\xi )_{T})\text{,}
\end{equation*}%
then%
\begin{equation*}
u(x+\tilde{B}(\xi )_{t},T-t)=Y(\xi )_{t}\ \ 
\end{equation*}%
and%
\begin{equation*}
\nabla u(x+\tilde{B}(\xi )_{t},T-t)=Z(\xi )_{t}\ 
\end{equation*}%
for all $t\leq T$ almost surely.
\end{theorem}

\begin{proof}
Define an equivalent probability measure $\boldsymbol{Q}$ by $\left. \frac{d%
\boldsymbol{Q}}{d\boldsymbol{P}}\right\vert _{\mathcal{F}_{t}}=\mathcal{E}%
(N(\xi ))_{t}$. Then $\tilde{B}(\xi )$ is a Brownian motion (up to time $T$)
under $\boldsymbol{Q}$. Let $\tilde{S}=S(\xi )-\langle N(\xi ),S(\xi
)\rangle $. According to Corollary \ref{coro-h01}, $(Y(\xi ),\tilde{S})$
solves the following BSDE%
\begin{equation*}
dY_{t}=\sum_{j=1}^{d}D_{\tilde{B}}(\tilde{S})_{t}^{j}f^{j}(Y_{t},D_{\tilde{B}%
}(\tilde{S})_{t})dt-g(t,Y_{t},D_{\tilde{B}}(\tilde{S})_{t})dt+d\tilde{S}_{t}
\end{equation*}%
with terminal value $Y_{T}=\varphi (x+\tilde{B}(\xi )_{T})$ on $(\Omega ,%
\mathcal{F},\mathcal{F}_{t},\boldsymbol{Q})$. This system is exactly the
BSDE that $Y_{t}=u(x+\tilde{B}(\xi )_{t},T-t)$ should satisfy, and the
conclusions follow immediately.
\end{proof}

By Theorem \ref{main-t31}, in order to provide a probabilistic
representation for (\ref{peq-01}), the problem is reduced to solve the
function equation 
\begin{equation}
\xi =u_{0}\left( x+B_{T}-\int_{0}^{T}f(Y(\xi )_{s},Z(\xi )_{s})ds\right)
:=\phi (\xi )\text{.}  \label{Fix11}
\end{equation}

The following local existence of the solutions to (\ref{Fix11}) is
elementary.

\begin{proposition}
\label{prop3.2} Let $C_{u_{0}}$ and $C_{f}$ be the Lipschitz constants for $%
u_{0}$ and $f$ respectively, and 
\begin{equation*}
\tau =\sqrt{\frac{1}{8C_{u_{0}}^{2}C_{f}^{2}}+1}-1\text{.}
\end{equation*}%
If $T\leq \tau $, then (\ref{Fix11}) admits a unique fixed point $\xi \in
L^{\infty }(\Omega ,\mathcal{F}_{T},\boldsymbol{P})$.
\end{proposition}

\begin{proof}
For $\xi $ and $\eta $ in $L^{2}(\Omega ,\mathcal{F}_{T},\boldsymbol{P})$, 
\begin{align*}
|\phi (\xi )-\phi (\eta )|& \leq C_{u_{0}}\int_{0}^{t}|f(Y(\xi )_{s},Z(\xi
)_{s})-f(Y(\eta )_{s},Z(\eta )_{s})|ds \\
& \leq C_{u_{0}}C_{f}\int_{0}^{t}|Y(\xi -\eta )_{s}|+|Z(\xi -\eta )_{s}|ds
\end{align*}%
and furthermore, 
\begin{align*}
\boldsymbol{E}\left[ |\phi (\xi )-\phi (\eta )|^{2}\right] & \leq
C_{u_{0}}^{2}C_{f}^{2}\boldsymbol{E}\left[ \left( \int_{0}^{T}|Y(\xi -\eta
)_{s}|+|Z(\xi -\eta )_{s}|ds\right) ^{2}\right] \\
& \leq 2C_{u_{0}}^{2}C_{f}^{2}T\left[ \int_{0}^{T}(\boldsymbol{E}|Y(\xi
-\eta )_{s}|^{2}+\boldsymbol{E}|Z(\alpha -\beta )_{s}|^{2})ds\right] \text{.}
\end{align*}%
Now, for any $\eta \in L^{2}(\Omega ,\mathcal{F}_{t},\mathbf{P})$, 
\begin{equation*}
\Vert Y(\eta )_{s}\Vert _{2}^{2}=\boldsymbol{E}[|\boldsymbol{E}(\eta |%
\mathcal{F}_{s})|^{2}]\leq \Vert \eta \Vert _{2}^{2}\text{,}
\end{equation*}%
and 
\begin{align*}
\boldsymbol{E}\left[ \int_{0}^{t}|Z(\eta )_{s}|^{2}ds\right] & =\boldsymbol{E%
}\left[ |Y(\eta )_{t}-Y(\eta )_{0}|^{2}\right] \\
& \leq 2\Vert \eta \Vert _{2}^{2}\text{,}
\end{align*}%
so that 
\begin{align*}
\boldsymbol{E}\left[ |\phi (\xi )-\phi (\eta )|^{2}\right] & \leq
2C_{u_{0}}^{2}C_{f}^{2}T\left[ \int_{0}^{T}\Vert \xi -\eta \Vert
_{2}^{2}ds+2\Vert \xi -\eta \Vert _{2}^{2}\right] \\
& =2C_{u_{0}}^{2}C_{f}^{2}T(T+2)\Vert \xi -\eta \Vert _{2}^{2}\text{.}
\end{align*}%
Hence 
\begin{equation*}
||\phi (\xi )-\phi (\eta )||_{2}\leq \sqrt{2}C_{u_{0}}C_{f}\sqrt{T(T+2)}%
||\xi -\eta ||_{2}
\end{equation*}%
and the claim follows from a simple application of the fixed point theorem.
\end{proof}

\begin{theorem}
\label{fu-th1}Under the previous assumptions, the function equation (\ref%
{Fix11}) has a unique solution.
\end{theorem}

\begin{proof}
We try to extend the local solution constructed in the previous proposition
to the case $T>\tau $. To do this we need a uniform gradient estimate of the
solutions to PDE (\ref{peq-01}): there exists a constant $C_{u}$ such that 
\begin{equation*}
\sup_{(t,x)\in \lbrack 0,T]\times \boldsymbol{R}^{d}}|\nabla u|(x,t)\leq
C_{u}\text{.}
\end{equation*}%
For the proof see for example \cite{MR0241822} and \cite{MR2053051}. Based
on such constant $C_{u}$ we define 
\begin{equation*}
\tau ^{\prime }=\sqrt{\frac{1}{8C_{u}^{2}C_{f}^{2}}+1}-1\text{.}
\end{equation*}%
Now we construct the random variable $\xi ^{\prime }$ on the time interval $%
[\tau ,\tau +\tau ^{\prime }]$. Consider the following function equation: 
\begin{equation}
\xi =u\left( x+B_{\tau +\tau ^{\prime }}^{\prime }-\int_{\tau }^{\tau +\tau
^{\prime }}f(s,Y(\xi )_{s},Z(\xi )_{s})ds,\tau \right) :=\phi ^{\prime }(\xi
).  \label{Fix1112}
\end{equation}%
where $B^{^{\prime }}$ is a Brownian motion on $[\tau ,\tau +\tau ^{\prime
}] $ defined by $B_{s}^{^{\prime }}=B_{s}-B_{\tau }$ for $s\in \lbrack \tau
,\tau +\tau ^{\prime }]$. Analogous to the proof of Proposition \ref{prop3.2}%
, (\ref{Fix1112}) admits a unique fixed point $\xi ^{\prime }\in L^{\infty
}(\Omega ,\mathcal{F}_{\tau +\tau ^{\prime }},\boldsymbol{P})\subset
L^{2}(\Omega ,\mathcal{F}_{\tau +\tau ^{\prime }},\boldsymbol{P})$. Based on
such $\xi ^{\prime }$, we have the following representation formulae on $%
[\tau ,\tau +\tau ^{\prime }]$: 
\begin{equation*}
u(x+\tilde{B}(\xi ^{\prime })_{t},2\tau +\tau ^{\prime }-t)=Y(\xi ^{\prime
})_{t}
\end{equation*}%
and 
\begin{equation*}
\nabla u(x+\tilde{B}(\xi ^{\prime })_{t},2\tau +\tau ^{\prime }-t)=Z(\xi
^{\prime })_{t}
\end{equation*}%
for $t\in \lbrack \tau ,\tau +\tau ^{\prime }]$, where 
\begin{equation*}
\tilde{B}(\xi ^{\prime })_{t}=B_{t}^{\prime }-\int_{\tau }^{t}f(u,Y(\xi
^{\prime })_{u},Z(\xi ^{\prime })_{u})du.
\end{equation*}

We then move to the next interval $[\tau +\tau ^{\prime },\tau +2\tau
^{\prime }]$ and repeat the above procedure until we touch $T$.
\end{proof}

\section{Some applications}

In this section we establish some explicit gradient estimates for the
solution of (\ref{peq-01}) by using the representation theorem \ref{main-t31}%
, to demonstrate the usefulness of non-linear Cameron-Martin's formula.
Further applications will appear in a separate paper.

Let us retain the notations and assumptions established in the previous
section. Let $T>0$ be fixed, $u=(u^{1},\cdots ,u^{m})$ be the unique
solution to the initial value problem (\ref{peq-01}) with initial data $%
u_{0} $, where $f$ \ (which determines the nature of the quadratic
non-linear term) depends only on $y$, i.e. $f(y,z)$ does not depend on $z$
and $g=0$. That is, $u$ is the solution to the initial value problem of the
following system of quasi-linear parabolic equations%
\begin{equation}
\frac{\partial u^{i}}{\partial t}+\sum_{j=1}^{d}f^{j}(u)\frac{\partial u^{i}%
}{\partial x^{j}}=\frac{1}{2}\Delta u^{i}\text{ \ \ in }\mathbb{R}^{d}\text{ 
}  \label{p-l-1}
\end{equation}%
for $i=1,\cdots ,m$.

Let $\xi $ be the solution to the function equation (\ref{Fix11})
established in Theorem \ref{fu-th1}.

Let $\boldsymbol{Q}$ be the equivalent measure with density process $\left. 
\frac{d\boldsymbol{Q}}{d\boldsymbol{P}}\right\vert _{\mathcal{F}_{t}}=R_{t}=%
\mathcal{E}(N(\xi ))_{t}$, where 
\begin{equation*}
N(\xi )_{t}=\int_{0}^{t}f(Y(\xi ))_{s}).dB_{s}\text{ .}
\end{equation*}

\begin{theorem}
\label{ma-es1}1) Let $p\in \lbrack 1,2)$ and $(P_{t})_{t\geq 0}$ the heat
semi-group, i.e. $P_{t}=e^{\frac{1}{2}t\Delta }$. Then 
\begin{equation*}
\int_{0}^{t}P_{s}\left\vert \nabla u^{i}\right\vert ^{p}\left( x,T-s\right)
ds\leq d^{1-\frac{p}{2}}\left( E^{\boldsymbol{P}}\langle Y(\xi )^{i}\rangle
_{t}\right) ^{\frac{p}{2}}\left( \int_{0}^{t}E^{\boldsymbol{P}}R_{s}^{\frac{2%
}{2-p}}ds\right) ^{1-\frac{p}{2}}
\end{equation*}%
for any $t\leq T$, $i=1,\cdots ,m$.

2) We have%
\begin{equation*}
\left\vert \nabla u^{i}\right\vert ^{2}(x,T)\leq \overline{\lim }%
_{t\downarrow 0}\frac{1}{t}E^{\boldsymbol{P}}\langle Y(\xi )^{i}\rangle _{t}
\end{equation*}%
for any $t\leq T$, $i=1,\cdots ,m$.
\end{theorem}

\begin{proof}
By Theorem \ref{main-t31} 
\begin{equation*}
Y(\xi )_{t}^{i}-Y(\xi )_{0}^{i}=\sum_{j=1}^{d}\int_{0}^{t}\frac{\partial
u^{i}}{\partial x^{j}}(x+\tilde{B}(\xi )_{s},T-s)dB_{s}^{j}
\end{equation*}%
so that%
\begin{equation*}
\langle Y(\xi )^{i}\rangle _{t}=\int_{0}^{t}\left\vert \nabla
u^{i}\right\vert ^{2}(x+\tilde{B}(\xi )_{s},T-s)ds
\end{equation*}%
and therefore%
\begin{equation*}
E^{\boldsymbol{P}}\int_{0}^{t}\left\vert \nabla u^{i}\right\vert ^{2}(x+%
\tilde{B}(\xi )_{s},t-s)ds=E^{\boldsymbol{P}}\langle Y(\xi )^{i}\rangle _{t}%
\text{.}
\end{equation*}%
On the other hand, for $1\leq p<2$ we have%
\begin{eqnarray*}
&&E^{\boldsymbol{Q}}\left\{ \int_{0}^{t}\left\vert \nabla u^{i}\right\vert
^{p}(x+\tilde{B}(\xi )_{s},T-s)ds\right\} \\
&=&\int_{0}^{t}E^{\boldsymbol{Q}}\left\{ \left\vert \nabla u^{i}\right\vert
^{p}(x+\tilde{B}(\xi )_{s},T-s)\right\} ds \\
&\leq &\int_{0}^{t}\left( E^{\boldsymbol{Q}}\left\{ \frac{1}{R_{s}}%
\left\vert \nabla u^{i}\right\vert ^{2}(x+\tilde{B}(\xi )_{s},T-s)\right\}
\right) ^{\frac{p}{2}}\left( E^{\boldsymbol{Q}}\left( R_{s}^{\frac{p}{2-p}%
}\right) \right) ^{1-\frac{p}{2}}ds \\
&=&\int_{0}^{t}\left( E^{\boldsymbol{P}}\left\{ \left\vert \nabla
u^{i}\right\vert ^{2}(x+\tilde{B}(\xi )_{s},T-s)\right\} \right) ^{\frac{p}{2%
}}\left( E^{\boldsymbol{P}}\left( R_{s}^{\frac{2}{2-p}}\right) \right) ^{1-%
\frac{p}{2}}ds
\end{eqnarray*}%
Since $\tilde{B}(\xi )$ is a Brownian motion under $\boldsymbol{Q}$, so that 
\begin{eqnarray*}
&&E^{\boldsymbol{Q}}\int_{0}^{t}\left\vert \nabla u^{i}\right\vert ^{p}(x+%
\tilde{B}(\xi )_{s},T-s)ds \\
&=&\int_{0}^{t}P_{s}\left( \left\vert \nabla u^{i}\right\vert ^{p}(\cdot
,T-s)\right) (x)ds \\
&\leq &\int_{0}^{t}\left( E^{\boldsymbol{P}}\left\vert \nabla
u^{i}\right\vert ^{2}(x+\tilde{B}(\xi )_{s},T-s)\right) ^{\frac{p}{2}}\left(
E^{\boldsymbol{P}}R_{s}^{\frac{2}{2-p}}\right) ^{1-\frac{p}{2}}ds \\
&\leq &d^{1-\frac{p}{2}}\left( \int_{0}^{t}E^{\boldsymbol{P}}\left\vert
\nabla u^{i}\right\vert ^{2}(x+\tilde{B}(\xi )_{s},T-s)ds\right) ^{\frac{p}{2%
}}\left( \int_{0}^{t}E^{\boldsymbol{P}}R_{s}^{\frac{2}{2-p}}ds\right) ^{1-%
\frac{p}{2}} \\
&=&d^{1-\frac{p}{2}}\left( E^{\boldsymbol{P}}\langle Y(\xi )^{i}\rangle
_{t}\right) ^{\frac{p}{2}}\left( \int_{0}^{t}E^{\boldsymbol{P}}R_{s}^{\frac{2%
}{2-p}}ds\right) ^{1-\frac{p}{2}}
\end{eqnarray*}%
which is the first estimate. To prove the second one, \ we write the
previous estimate as 
\begin{eqnarray*}
&&\frac{1}{t}\int_{0}^{t}P_{s}\left\vert \nabla u^{i}\right\vert
^{p}(x,T-s)ds \\
&\leq &d^{1-\frac{p}{2}}\left( \frac{1}{t}E^{\boldsymbol{P}}\langle Y(\xi
)^{i}\rangle _{t}\right) ^{\frac{p}{2}}\left( \frac{1}{t}\int_{0}^{t}E^{%
\boldsymbol{P}}R_{s}^{\frac{2}{2-p}}ds\right) ^{1-\frac{p}{2}}\text{.}
\end{eqnarray*}%
Letting $t\rightarrow 0$ one obtains that%
\begin{equation*}
\sqrt[p]{\left\vert \nabla u^{i}\right\vert ^{p}(x,T)}\leq d^{\frac{1}{p}-%
\frac{1}{2}}\overline{\lim }_{t\rightarrow 0}\sqrt{\frac{1}{t}E^{\boldsymbol{%
P}}\langle Y(\xi )^{i}\rangle _{t}}\text{.}
\end{equation*}%
then letting $p\uparrow 2$ we obtain 2).
\end{proof}

\begin{lemma}
\label{lem-grad}Then for any $p\in \lbrack 1,2)$%
\begin{equation*}
E^{\boldsymbol{P}}\left( R_{t}^{\frac{2}{2-p}}\right) \leq \exp \left\{ 
\frac{p}{(2-p)^{2}}t\max_{|y|\leq |u_{0}|_{\infty }}|f(y)|^{2}\right\} \text{
.}
\end{equation*}
\end{lemma}

\begin{proof}
Let%
\begin{equation*}
H_{t}=\exp \left[ \frac{2}{(2-p)}\int_{0}^{t}f(Y(\xi )_{s}).dB_{s}-\frac{2}{%
(2-p)^{2}}\int_{0}^{t}|f(Y(\xi )_{s})|^{2}ds\right]
\end{equation*}%
which is exponential martingale, so that $E^{\boldsymbol{P}}\left(
H_{t}\right) =1$. Then 
\begin{eqnarray*}
R_{t}^{\frac{2}{2-p}} &=&H_{t}\exp \left[ \frac{p}{(2-p)^{2}}%
\int_{0}^{t}|f(Y(\xi )_{s})|^{2}ds\right] \\
&\leq &H_{t}\exp \left[ \frac{p}{(2-p)^{2}}t\max_{|y|\leq |u_{0}|_{\infty
}}|f(y)|^{2}\right]
\end{eqnarray*}%
which yields the claim.
\end{proof}

\begin{corollary}
We have%
\begin{equation}
\int_{0}^{t}P_{s}\left\vert \nabla u^{i}\right\vert ^{p}(x,T-s)ds\leq d^{1-%
\frac{p}{2}}|u_{0}^{i}|_{\infty }^{p}\exp \left\{ \frac{p}{2(2-p)}%
t\max_{|y|\leq |u_{0}|_{\infty }}|f(y)|^{2}\right\}  \label{grad-01e}
\end{equation}%
for any $p\in \lbrack 1,2)$, and $i=1,\cdots ,m$, where $P_{s}=e^{\frac{1}{2}%
s\Delta }$ the heat semigroup.
\end{corollary}

\begin{proof}
Observe that%
\begin{eqnarray*}
E^{\boldsymbol{P}}\langle Y(\xi _{t}^{i})\rangle _{s} &=&E^{\boldsymbol{P}}%
\left[ E^{\boldsymbol{P}}(\xi _{t}^{i}|\mathcal{F}_{s})-E^{\boldsymbol{P}%
}(\xi _{t}^{i}|\mathcal{F}_{0})\right] ^{2} \\
&=&E^{\boldsymbol{P}}\left[ E^{\boldsymbol{P}}(\xi _{t}^{i}|\mathcal{F}%
_{s})^{2}-E^{\boldsymbol{P}}(\xi _{t}^{i}|\mathcal{F}_{0})^{2}\right] \\
&\leq &||\xi _{t}^{i}||^{2} \\
&\leq &|u_{0}^{i}|_{\infty }
\end{eqnarray*}%
and therefore, the item 1) in Theorem \ref{ma-es1} together with Lemma \ref%
{lem-grad} yields the gradient estimate \ (\ref{grad-01e}).
\end{proof}

\section{Non-linear flow associated with quasi-linear PDEs}

In this section we construct a non-linear stochastic flow associated with
the quasi-linear system (\ref{p-l-1}).

Assume that $f^{j}:\mathbb{R}^{m}\rightarrow \mathbb{R}^{d}$ ($j=1,\cdots ,d$%
) are differentiable with bounded 1st, 2nd and 3rd derivatives, and the
initial value $u^{i}(x,0)=u_{0}^{i}(x)$ ($i=1,\cdots ,m$) are bounded,
differentiable with 1st and 2nd bounded derivatives:%
\begin{equation*}
|\nabla ^{k-1}u_{0}|\leq C_{0}\text{, \ \ }|\nabla ^{k}f|\leq C_{1}\text{ \
for }k=1,2,3
\end{equation*}%
for some non-negative constants $C_{0}$ and $C_{1}$, here $\nabla ^{k}$
denote the $k$-th derivative in space variables. Of course $\nabla
^{0}u_{0}=u_{0}$. In general, we apply $\nabla $ to mean the total
derivative operator in space variables. For example, $\nabla u$ means $(%
\frac{\partial u^{i}}{\partial x^{j}})$ but does not include the derivative
in time parameter $t$.

We are going to construct a continuous adapted process $\xi =(\xi _{t})$
with values in the function space $C_{b}^{2}(\mathbb{R}^{d},\mathbb{R}^{m})$
such that $\nabla ^{j}u(\cdot ,t)=E^{\boldsymbol{P}}(\nabla ^{j}\xi _{t})$
at least for small $t$. The spirit in devising such a formula is quite
similar to those initiated in the seminal works Bismut \cite{MR755001} and
Malliavin \cite{MR1450093}.

Consider the function space $L^{\infty }(\Omega ;C([0,T];C_{b}^{2}(\mathbb{R}%
^{d},\mathbb{R}^{m})))$. If 
\begin{equation*}
\xi \in L^{\infty }(\Omega ;C([0,T];C_{b}^{2}(\mathbb{R}^{d},\mathbb{R}%
^{m})))
\end{equation*}%
then for any $\omega $, $\xi (\omega )\in C([0,T];C_{b}^{2}(\mathbb{R}^{d},%
\mathbb{R}^{m}))$, so that $\xi _{t}(\omega )\in C_{b}^{2}(\mathbb{R}^{d},%
\mathbb{R}^{m})$ and $t\rightarrow \xi _{t}(\omega )$ is continuous from $%
[0,T]$ to $C_{b}^{2}(\mathbb{R}^{d},\mathbb{R}^{m})$, and $x\rightarrow \xi
_{t}(\omega ,x)$ has continuous and bounded first and second derivatives.
Let $\boldsymbol{H}_{T}$ denote the space of all functions $\xi $ in $%
L^{\infty }(\Omega ;C([0,T];C_{b}^{2}(\mathbb{R}^{d},\mathbb{R}^{m})))$ such
that $\xi =(\xi _{t})_{t\in \lbrack 0,T]}$ is adapted. $\boldsymbol{H}_{T}$
is equipped with the $L^{\infty }$-norm, namely%
\begin{equation*}
||\xi ||=\text{ess}\sup_{\omega \in \Omega }\sup_{t\in \lbrack
0,T]}\sum_{j=0}^{2}\sup_{x\in \boldsymbol{R}^{d}}|\nabla ^{j}\xi
_{t}(x,\omega )|
\end{equation*}%
where, as we have explained, 
\begin{equation*}
\nabla \xi _{t}(x,\omega )=\left( \left. \frac{\partial \xi ^{i}}{\partial
x^{j}}\right\vert _{(x,\omega ,t)}\right) _{\substack{ i=1,\cdots ,m \\ %
j=1,\cdots ,d}}
\end{equation*}%
etc. $\boldsymbol{H}_{T}$ is a Banach space under $||\cdot ||$.

In this section, if $\zeta =(\zeta ^{i})$, where $\zeta ^{i}\in L^{2}(\Omega
,\mathcal{F},\boldsymbol{P})$, then we define $Y(\zeta )_{s}=E^{\boldsymbol{P%
}}\{\zeta |\mathcal{F}_{s}\}$ and $Z(\zeta )=D_{\boldsymbol{B}}(Y(\zeta ))$.
Therefore, for any $t>0$, if $\zeta ^{i}\in L^{2}(\Omega ,\mathcal{F}_{t},%
\boldsymbol{P})$, then $(Y(\zeta )_{s},Z(\zeta )_{s})$ is the unique
solution of the BSDE%
\begin{equation*}
dY(\zeta )_{s}=Z(\zeta )_{s}.dB_{s}\text{, \ }Y(\zeta )_{t}=\zeta \text{.}
\end{equation*}%
It is easy to see that, if $\xi \in \boldsymbol{H}_{T}$, then $Y(\nabla
^{k}\xi _{t})=\nabla ^{k}Y(\xi _{t})$ and $Z(\nabla ^{k}\xi _{t})=\nabla
^{k}Z(\xi _{t})$ for any $t\leq T$ and $k=0,1,2$. This follows from the fact
that the mappings $\zeta \rightarrow Y(\zeta )$ and $\zeta \rightarrow
Z(\zeta )$ are both affine.

Let $\xi \in \boldsymbol{H}_{T}$. Then, according to the non-linear
Cameron-Martin formula, for any fixed $t\leq T$, we define a probability $%
\boldsymbol{Q}_{t,x}\boldsymbol{\ }$on $(\Omega ,\mathcal{F}_{t})$ by%
\begin{equation*}
\left. \frac{d\boldsymbol{Q}_{t,x}}{d\boldsymbol{P}}\right\vert _{\mathcal{F}%
_{t}}=\exp \left[ \int_{0}^{t}f(Y(\xi _{t}(\cdot ,x))_{s}).dB_{s}-\frac{1}{2}%
\int_{0}^{t}|f|^{2}(Y(\xi _{t}(\cdot ,x))_{s})ds\right] \text{.}
\end{equation*}%
We will omit the argument $\cdot $ (a sample point) and the space variable $%
x $ if no confusion may arise. Under $\boldsymbol{Q}_{t,x}$,\ $\tilde{B}(\xi
_{t})_{s}=B_{s}-\int_{0}^{s}f(Y(\xi _{t})_{r})dr$ ($0\leq s\leq t$) is
Brownian motion up to time $t$. Define 
\begin{eqnarray*}
X(\xi )_{t} &=&x+\tilde{B}(\xi _{t})_{t} \\
&=&x+B_{t}-\int_{0}^{t}f(Y(\xi _{t})_{s})ds\text{ \ \ \ \ for }t\leq T\text{.%
}
\end{eqnarray*}%
According to Theorem \ref{main-t31}, we want to find a fixed point $\xi \in 
\boldsymbol{H}_{T}$: $\xi =u_{0}(X(\xi ))$. To this end we define $\Phi (\xi
_{\cdot })=u_{0}(X(\xi )_{\cdot })$ for any $\xi \in \boldsymbol{H}_{T}$.
Then%
\begin{equation*}
\nabla \Phi (\xi )_{t}=\nabla u_{0}(X(\xi )_{t})\nabla X(\xi )_{t}
\end{equation*}%
and%
\begin{eqnarray*}
\nabla ^{2}\Phi (\xi )_{t} &=&\nabla ^{2}u_{0}(X(\xi )_{t})(\nabla X(\xi
)_{t}\text{,}\nabla X(\xi )_{t}) \\
&&+\nabla u_{0}(X(\xi )_{t})\nabla ^{2}X(\xi )_{t}\text{.}
\end{eqnarray*}

\begin{lemma}
\label{lem-u1}1)For any $T>0$ and $\xi \in \boldsymbol{H}_{T}$ 
\begin{equation*}
||\Phi (\xi )||\leq C_{0}\left( 1+d\right) ^{2}+C_{0}C_{1}T\left\{
(2d+1)+(1+C_{1}T)||\xi ||\right\} ||\xi ||\text{.}
\end{equation*}%
2) If $K=2C_{0}\left( 1+d\right) ^{2}$ and%
\begin{equation*}
T\leq \frac{1}{2\sqrt{C_{0}C_{1}}\sqrt{d+1/2+C_{0}(1+C_{1})\left( 1+d\right)
^{2}}}\wedge 1\text{,}
\end{equation*}%
then $||\Phi (\xi )||\leq K$ as long as $\xi \in \boldsymbol{H}_{T}$ and $%
||\xi ||\leq K$.
\end{lemma}

\begin{proof}
By definition for any $t\leq T$ and any $x$ (but the argument $x$ is
suppressed from the notations for simplicity, and $|\cdot |_{\infty }$
denotes the essential supremum norm)%
\begin{equation*}
\nabla X(\xi )_{t}=I_{\mathbb{R}^{d}}-\int_{0}^{t}\nabla f(Y(\xi
_{t})_{s})Y(\nabla \xi _{t})_{s}ds
\end{equation*}%
and%
\begin{eqnarray*}
\nabla ^{2}X(\xi )_{t} &=&-\int_{0}^{t}\nabla ^{2}f(Y(\xi _{t})_{s})\left(
Y(\nabla \xi _{t})_{s},Y(\nabla \xi _{t})_{s}\right) ds \\
&&-\int_{0}^{t}\nabla f(Y(\xi _{t})_{s})Y(\nabla ^{2}\xi _{t})_{s}ds\text{.}
\end{eqnarray*}%
From these equations and the fact that the conditional expectation is a
contraction on $L^{\infty }(\Omega ,\mathcal{F}_{t},\boldsymbol{P})$, one
can easily to see the following estimates:%
\begin{equation*}
|\nabla X(\xi )_{t}|\leq d+C_{1}t|\nabla \xi _{t}|_{\infty }\text{,}
\end{equation*}%
and%
\begin{equation*}
|\nabla ^{2}X(\xi )_{t}|\leq C_{1}t|\nabla \xi _{t}|^{2}+C_{1}t|\nabla
^{2}\xi _{t}|\text{.}
\end{equation*}%
Therefore%
\begin{equation*}
|\Phi (\xi )|\leq C_{0}\text{,}
\end{equation*}%
\begin{eqnarray*}
|\nabla \Phi (\xi )_{t}| &\leq &C_{0}|\nabla X(\xi )_{t}| \\
&\leq &C_{0}\left\{ d+C_{1}t|\nabla \xi _{t}|_{\infty }\right\}
\end{eqnarray*}%
and%
\begin{eqnarray*}
|\nabla ^{2}\Phi (\xi )_{t}| &\leq &C_{0}|\nabla X(\xi
)_{t}|^{2}+C_{0}|\nabla _{x}^{2}X(\xi )_{t}| \\
&\leq &C_{0}\left\{ d+C_{1}t|\nabla \xi _{t}|_{\infty }\right\}
^{2}+C_{0}C_{1}t\left[ |\nabla \xi _{t}|_{\infty }^{2}+|\nabla ^{2}\xi
_{t}|_{\infty }\right]
\end{eqnarray*}%
which yield the required estimates.
\end{proof}

\begin{lemma}
\label{lem-u2}Let $T>0$. There is positive constant depending only on $%
C_{0},C_{1}$ and $d$ such that%
\begin{equation*}
||\Phi (\xi )-\Phi (\eta )||\leq K_{0}C_{1}T(1+T)\left( 1+||\xi ||+||\eta
||\right) ||\xi -\eta ||
\end{equation*}%
for any $\xi ,\eta \in \boldsymbol{H}_{T}$.
\end{lemma}

\begin{proof}
By a simple computation, we have%
\begin{eqnarray*}
|X(\xi )_{t}-X(\eta )_{t}| &\leq &\left\vert \int_{0}^{t}f(Y(\xi _{t}-\eta
_{t})_{s})ds\right\vert \\
&\leq &tC_{1}|\xi _{t}-\eta _{t}|_{\infty }\text{,}
\end{eqnarray*}%
\begin{eqnarray*}
|\nabla X(\xi )_{t}-\nabla X(\eta )_{t}| &\leq &\left\vert
\int_{0}^{t}\left( \nabla f(Y(\xi _{t})_{s})-\nabla f(Y(\eta
_{t})_{s})\right) Y(\nabla \xi _{t})_{s}ds\right\vert \\
&&+\left\vert \int_{0}^{t}\nabla f(Y(\eta _{t})_{s})Y(\nabla (\xi _{t}-\eta
_{t}))_{s}ds\right\vert \\
&\leq &C_{1}\int_{0}^{t}Y(\xi _{t}-\eta _{t})_{s}||Y(\nabla \xi _{t})_{s}|ds
\\
&&+C_{1}\int_{0}^{t}|Y(\nabla (\xi _{t}-\eta _{t}))_{s}|ds \\
&\leq &C_{1}t|\xi _{t}-\eta _{t}||\nabla \xi _{t}|+C_{1}t|\nabla \xi
_{t}-\nabla \eta _{t}|\text{,}
\end{eqnarray*}%
and, similarly%
\begin{eqnarray*}
\left\vert \nabla ^{2}X(\xi )_{t}-\nabla ^{2}X(\eta )_{t}\right\vert &\leq
&C_{1}t|\xi _{t}-\eta _{t}||\nabla \xi _{t}|^{2} \\
&&+C_{1}t\left( |\nabla \xi _{t}|+|\nabla \eta _{t}|\right) |\nabla \xi
_{t}-\nabla \eta _{t}| \\
&&+C_{1}t|\xi _{t}-\eta _{t}||\nabla ^{2}\xi _{t}|+C_{1}t|\nabla ^{2}\xi
_{t}-\nabla ^{2}\eta _{t}|
\end{eqnarray*}%
and the estimate follows easily from these inequalities.$\Phi (\xi _{\cdot
})=u_{0}(X(\xi )_{\cdot })$%
\begin{equation*}
|\Phi (\xi )_{t}-\Phi (\eta )_{t}|\leq C_{0}C_{1}t|\xi _{t}-\eta
_{t}|_{\infty }
\end{equation*}%
\begin{eqnarray*}
|\nabla \Phi (\xi )_{t}-\nabla \Phi (\eta )_{t}| &=&C_{0}|X(\xi )_{t}-X(\eta
)_{t}||\nabla X(\xi )_{t}| \\
&&+C_{0}|\nabla X(\xi )_{t}-\nabla X(\eta )_{t}| \\
&\leq &C_{0}^{2}C_{1}t|\xi _{t}-\eta _{t}|_{\infty }\left( d+C_{1}t|\nabla
\xi _{t}|_{\infty }\right) \\
&&+C_{0}C_{1}t\left( |\xi _{t}-\eta _{t}||\nabla \xi _{t}|+|\nabla \xi
_{t}-\nabla \eta _{t}|\right)
\end{eqnarray*}
\end{proof}

Now we are in a position to establish a Bismut type formula (see Bismut \cite%
{MR755001} for the linear case) for the solution of quasi-linear system (\ref%
{p-l-1}).

\begin{theorem}
\label{th-bis1}There is $T>0$ depending only on $d$, $C_{0}$ and $C_{2}$, so
that there is a unique fixed point $\xi $ of $\Phi $ in $\boldsymbol{H}_{T}$%
. Moreover%
\begin{equation}
\nabla ^{j}u(x,t)=E^{\boldsymbol{P}}(\nabla ^{j}\xi _{t}(\cdot ,x))\text{,}\
\ \ j=0,1,2\text{.}  \label{nis-r3}
\end{equation}%
\begin{equation*}
u(x+\tilde{B}(\xi _{t})_{s},t-s)=Y(\xi _{t})_{s}\text{ \ for all }s\leq
t\leq T\text{ a.e.}
\end{equation*}%
and%
\begin{equation*}
\nabla u(x+\tilde{B}(\xi _{t})_{s},t-s)=Z(\xi _{t})_{s}\text{ \ for all }%
s\leq t\leq T\text{ a.e.}
\end{equation*}
\end{theorem}

\begin{proof}
According to Theorem \ref{main-t31}, for $t\leq T$ we have $u(x,t)=E^{%
\boldsymbol{P}}(\xi _{t}(\cdot ,x)|\mathcal{F}_{0})$. Taking expectation we
obtain $u(x,t)=E^{\boldsymbol{P}}(\xi _{t}(\cdot ,x))$. Since $\xi
_{t}(\omega ,\cdot )\in C_{b}^{2}(\mathbb{R}^{d},\mathbb{R}^{m})$ so we may
take derivatives under integration to obtain (\ref{nis-r3}).
\end{proof}

\vskip0.4truecm

\noindent {\textbf{Acknowledgements.}} The research was supported in part by
EPSRC\ grant EP/F029578/1, and by the Oxford-Man Institute.

\bibliographystyle{amsplain}
\bibliography{2010qian}

\noindent {\small \textsc{G. Liang, {\small A. Lionnet \textsc{and} }Z. Qian 
}}

\noindent{\small Mathematical Institute and Oxford-Man Institute}

\noindent{\small University of Oxford}

\noindent{\small Oxford OX1 3LB, England}

\vskip0.3truecm

\noindent {\small {Email: {\small \texttt{liangg@maths.ox.ac.uk,
arnaud.lionnet@maths.ox.ac.uk}}\texttt{\  \newline
and qianz@maths.ox.ac.uk}}}

\end{document}